\documentclass[reqno, 11pt]{amsart}
\makeatletter
\let\origsection=\section %save orignal definition
\def\section{\@ifstar{\origsection*}{\mysection}}
\def\mysection{\@startsection{section}{1}\z@{.7\linespacing\@plus\linespacing}{.5\linespacing}{\normalfont\scshape\centering\S}}
\makeatother
\usepackage{amsmath,amssymb,amsthm}
\usepackage{mathrsfs}
\usepackage{dsfont}
\usepackage{mathabx}\changenotsign

\usepackage{microtype}

\usepackage{xcolor}
\usepackage[backref=page]{hyperref}
\hypersetup{%
    colorlinks,
    linkcolor={red!50!black},
    citecolor={green!50!black},
    urlcolor={blue!50!black}
}

\usepackage{varioref}

\usepackage[english,algoruled,vlined,resetcount,linesnumbered]{
algorithm2e}%,linesnumbered

\usepackage{bookmark}

\usepackage[abbrev,msc-links,backrefs]{amsrefs}
\usepackage{doi}

\renewcommand{\PrintDOI}[1]{\doi{#1}}

\usepackage[T1]{fontenc}
\usepackage{lmodern}

\usepackage[english]{babel}

\usepackage{tikz}

\usepackage{fullpage}
\usepackage{setspace}
\makeatletter
\def\@setdate{\datename\ \@date}
\def\@setaddresses{\par
  \nobreak \begingroup
\footnotesize
  \def\author##1{\nobreak\addvspace\bigskipamount}%
  \def\\{\unskip, \ignorespaces}%
  \interlinepenalty\@M
  \def\address##1##2{\begingroup
    \par\addvspace\bigskipamount\indent
    \@ifnotempty{##1}{(\ignorespaces##1\unskip) }%
    {\scshape\ignorespaces##2}\par\endgroup}%
  \def\curraddr##1##2{\begingroup
    \@ifnotempty{##2}{\nobreak\indent{\itshape Current address}%
      \@ifnotempty{##1}{, \ignorespaces##1\unskip}\/:\space
      ##2\par}\endgroup}%
  \def\email##1##2{\begingroup
    \@ifnotempty{##2}{\nobreak\indent{\itshape Email addresses}%
      \@ifnotempty{##1}{, \ignorespaces##1\unskip}\/:\space
      \ttfamily##2\par}\endgroup}%
  \def\urladdr##1##2{\begingroup
    \@ifnotempty{##2}{\nobreak\indent{\itshape URL}%
      \@ifnotempty{##1}{, \ignorespaces##1\unskip}\/:\space
      \ttfamily##2\par}\endgroup}%
  \addresses
  \endgroup
}
\makeatother
\usepackage{enumitem}

\let\polishlcross=\l
\def\l{\ifmmode\ell\else\polishlcross\fi}

\renewcommand{\setminus}{\smallsetminus}

\makeatletter
\def\moverlay{\mathpalette\mov@rlay}
\def\mov@rlay#1#2{\leavevmode\vtop{%
   \baselineskip\z@skip \lineskiplimit-\maxdimen
   \ialign{\hfil$\m@th#1##$\hfil\cr#2\crcr}}}
\newcommand{\charfusion}[3][\mathord]{
    #1{\ifx#1\mathop\vphantom{#2}\fi
        \mathpalette\mov@rlay{#2\cr#3}
      }
    \ifx#1\mathop\expandafter\displaylimits\fi}
\makeatother

\let\epsilon\varepsilon

\newtheoremstyle{case}{}{}{\normalfont}{}{\itshape}{:}{ }{}

\newtheorem{thm}{Theorem}
\newtheorem{lem}[thm]{Lemma}

\theoremstyle{definition}
\newtheorem{defn}[thm]{Definition}

\newtheorem{ex}[thm]{Example}

\newtheorem{rem}[thm]{Remark}

\newtheoremstyle{case}{}{}{\normalfont}{}{\itshape}{\normalfont:}{ }{}

\theoremstyle{case}

\def\ex{\text{\rm ex}}
\def\forb{\text{\rm Forb}}
\let\:\colon
\let\subset\subseteq

\def\({\left(}
\def\){\right)}
\def\[{\left[}
\def\]{\right]}
\let\<\langle
\let\>\rangle

\def\llceil{\left\lceil}
\def\rrceil{\right\rceil}

\DeclareMathOperator{\im}{{\rm im}}
\usepackage{datetime}
\usepackage{lineno}
\newcommand*\patchAmsMathEnvironmentForLineno[1]{%
\expandafter\let\csname old#1\expandafter\endcsname\csname #1\endcsname
\expandafter\let\csname oldend#1\expandafter\endcsname\csname end#1\endcsname
\renewenvironment{#1}%
{\linenomath\csname old#1\endcsname}%
{\csname oldend#1\endcsname\endlinenomath}}% 
\newcommand*\patchBothAmsMathEnvironmentsForLineno[1]{%
\patchAmsMathEnvironmentForLineno{#1}%
\patchAmsMathEnvironmentForLineno{#1*}}%
\AtBeginDocument{%
\patchBothAmsMathEnvironmentsForLineno{equation}%
\patchBothAmsMathEnvironmentsForLineno{align}%
\patchBothAmsMathEnvironmentsForLineno{flalign}%
\patchBothAmsMathEnvironmentsForLineno{alignat}%
\patchBothAmsMathEnvironmentsForLineno{gather}%
\patchBothAmsMathEnvironmentsForLineno{multline}%
}
\begin{document}
%\linenumbers 
\onehalfspace
%\doublespace

%TITLE, ETC
\title{On hypergraphs without loose cycles}
 \author[
 Han
 \and Kohayakawa
 ]
 {
 Jie Han
 \and Yoshiharu Kohayakawa
  }

\shortdate
\yyyymmdddate
\settimeformat{ampmtime}
\date{\today, \currenttime}
\footskip=28pt

\address{Instituto de Matem\'atica e Estat\'{\i}stica, Universidade de
    S\~ao Paulo, S\~ao Paulo, Brazil}
\email{\{jhan|yoshi\}@ime.usp.br}

\thanks{%
  The first author is supported by FAPESP (2014/18641-5,
  2013/03447-6).  The second author is partially supported by FAPESP
  (2013/03447-6, 2013/07699-0), CNPq (310974/2013-5, 459335/2014-6)
  and NUMEC/USP (Project MaCLinC)}

\begin{abstract}
  Recently, Mubayi and Wang showed that for $r\ge 4$ and $\ell \ge 3$,
  the number of $n$-vertex $r$-graphs that do not contain any loose
  cycle of length $\ell$ is at most
  $2^{O( n^{r-1} (\log n)^{(r-3)/(r-2)})}$.  We improve this bound to
  $2^{O( n^{r-1} \log \log n) }$.
\end{abstract}

\maketitle

\section{Introduction}
Let two graphs $G$ and $H$ be given.  The graph~$G$ is called
\emph{$H$-free} if it does not contain any copy of $H$ as a subgraph.
One of the central problems in graph theory is to determine the
extremal and typical properties of $H$-free graphs on $n$~vertices.
For example, one of the first influential results of this type is the
Erd\H{o}s--Kleitman--Rothschild theorem~\cite{EKRo}, which, for
instance, implies that the number of triangle-free graphs with vertex
set $[n]=\{1,\dots,n\}$ is $2^{n^2/4+o(n^2)}$.  This has inspired a
great deal of work on counting the number of $H$-free graphs.  For an
overview of this line of research, the reader is referred to, e.g.,
\cite{bollobas98:_hered,proemel01:_asymp}.  For a recent, exciting
result in the area, see~\cite{morris16:_c}, which also contains a good
discussion of the general area, with several pointers to the
literature.  These problems are closely related to the so-called
\emph{Tur\'an problem}, which asks to determine the maximum possible
number of edges in an $H$-free graph.  More precisely, given an
$r$-uniform hypergraph (or $r$-graph) $H$, the \emph{Tur\'an number}
$\ex_r(n, H)$ is the maximum number of edges in an $r$-graph~$G$ on
$n$~vertices that is $H$-free.  Let $\forb_r(n, H)$ be the set of all
$H$-free $r$-graphs with vertex set $[n]$.  Noting that the subgraphs
of an $H$-free $r$-graph~$G$ are also $H$-free, we trivially see that
$|\forb_r(n, H)|\ge 2^{\ex_r(n, H)}$, by considering an $H$-free
$r$-graph~$G$ on~$[n]$ with the maximum number of edges and all its
subgraphs.  On the other hand for, fixed $r$ and $H$,
\begin{equation}
  \label{eq:1}
  |\forb_r(n, H)| \le \sum_{1\le i\le \ex_r(n, H)} \binom{\binom{n}{r}}{i} = 2^{O(\ex_r(n, H) \log n)}.  
\end{equation}
Hence the above simple bounds differ by a factor of $\log n$ in the exponent,
and all existing results support that this $\log n$ factor should be
unnecessary, i.e., the trivial lower bound should be closer to the truth.

There are very few results in the case $r>2$ and
$\ex_r(n, H) = o(n^r)$.  The only known case is when $H$ consists of
two edges sharing $t$ vertices~\cite{BDDLS, FrKu}.  Very recently,
Mubayi and Wang~\cite{MuWa} studied $|\forb_r(n, H)|$ when $H$ is a
loose cycle.  Given $\ell\ge 3$, an \emph{$r$-uniform loose cycle
  $C_\ell^r$} is an $\ell(r-1)$-vertex $r$-graph whose vertices can be
ordered cyclically in such a way that the edges are sets of
consecutive $r$ vertices and every two consecutive edges share exactly
one vertex.  When $r$ is clear from the context, we simply
write~$C_\ell$.  
% In this note, we sometimes simply say `cycle' instead
% of `loose cycle'.

\begin{thm}\cite{MuWa}\label{thm1}
  For every $\ell \ge 3$ and~$r\ge 4$, there exists
  $c=c(r,\ell)$ such that 
  \begin{equation}
    \label{eq:2}
    |\forb_r(n, C_{\ell})| < 2^{c n^{r-1} (\log n)^{(r-3)/(r-2)}}
  \end{equation}
  for all~$n$.  For $\ell \ge 4$ even, there exists $c=c(\ell)$ such
  that $|\forb_3(n, C_{\ell})| < 2^{c n^{2}}$ for all~$n$.
\end{thm}

Since $\ex_r(n, C_\ell) = \Omega (n^{r-1})$ for all
$r\ge 3$~\cite{KMV, FurediJiang}, Theorem~\ref{thm1} implies that
$|\forb_3(n, C_{\ell})| = 2^{\Theta(n^{2})}$ for even $\ell \ge 4$.
Mubayi and Wang also conjecture that similar results should hold for
$r=3$ and all~$\ell\geq3$ odd and for all $r\ge4$ and $\ell\ge3$,
i.e., $|\forb_r(n, C_{\ell})| = 2^{\Theta(n^{r-1})}$ for all such~$r$
and~$\ell$.  In this note we give the following improvement of
Theorem~\ref{thm1} for $r\ge 4$.

\begin{thm}\label{thm2}
  For every $\ell \ge 3$ and $r\ge 4$, we have
  \begin{equation}
    \label{eq:3}
    |\forb_r(n, C_{\ell})| < 2^{2r^2\ell n^{r-1}\log\log n}
  \end{equation}
  for all sufficiently large~$n$.
\end{thm}

%Throughout this note we omit floors and ceilings and 
In what follows, logarithms have base $2$.

\section{Edge-colored $r$-graphs}

Let $r\ge 2$ be an integer.  An \textit{$r$-uniform hypergraph}~$G$
(or \textit{$r$-graph}) on a \textit{vertex set}~$X$ is a collection
of $r$-element subsets of~$X$, called \textit{hyperedges} or simply
\textit{edges}.  The vertex set~$X$ of~$G$ is denoted~$V(G)$.  We
write~$e(G)$ for the number of edges in~$G$.  An \emph{$r$-partite
  $r$-graph~$G$} is an $r$-graph together with a vertex partition
$V(G)=V_1\cup \cdots \cup V_r$, such that every edge of~$G$ contains
exactly one vertex from each~$V_i$ ($i\in[n]$).  If all such edges are
present in~$G$, then we say that~$G$ is \emph{complete}.  We call an
$r$-partite $r$-graph \emph{balanced} if all parts in its vertex
partition have the same size.  Let $K_r(s)$ denote the complete
$r$-partite $r$-graph with $s$ vertices in each vertex class.

We now introduce some key definitions from~\cite{MuWa}, which are also
essential for us.  Given an $(r-1)$-graph $G$ with
$V(G)\subseteq [n]$, a \emph{coloring function} for~$G$ is a function
$\chi\:G\rightarrow [n]$ such that $\chi(e) = z_e\in [n]\setminus e$
for every $e\in G$.  We call $z_e$ the \emph{color} of~$e$.
%The vector of colors $N_G=(z_e)_{e\in G}$ is called an \emph{edge-coloring} of $G$.
The pair $(G,\chi)$ is an \emph{edge-colored} $(r-1)$-graph.

Given $G$, each edge-coloring $\chi$ of~$G$ gives an $r$-graph
$G^\chi=\big\{e\cup \{z_e\}\: e\in G\big\}$, called the
\emph{extension of $G$ by~$\chi$}.  When there is only one coloring
that has been defined, we write $G^*$ for~$G^\chi$.  Clearly any
subgraph $G'\subseteq G$ also admits an extension by~$\chi$, namely,
$(G')^\chi=\big\{e\cup \{z_e\}\: e\in G'\big\}\subseteq G^*$.  If
$G'\subseteq G$ and $\chi\restriction_{G'}$ is one-to-one and
$z_e\notin V(G')$ for all $e\in G'$, then $G'$ is called
\emph{strongly rainbow colored}.  We state the following simple remark
explicitly for later reference.

\begin{rem}
  \label{rem:strongrbw}
  A strongly rainbow colored copy of~$C_\ell^{r-1}$ in~$G'$ gives
  rise to a copy of $C_\ell^r$ in~$G^*$.
\end{rem}

The following definition is crucial.

\begin{defn}[$g_r(n,\ell)$, \cite{MuWa}]
 For $r\ge 4$ and $\ell\ge 3$, let $g_r(n,\ell)$ be the number of
 edge-colored $(r-1)$-graphs $G$ with $V(G)\subseteq [n]$ such that
 the extension~$G^*$ is $C_\ell^r$-free. 
\end{defn}

The function~$g_r(n,\ell)$ above counts the number of pairs~$(G,\chi)$
with~$G^\chi\in\forb_r(n,C_\ell^r)$.  Mubayi and Wang~\cite{MuWa}
proved that~$g_r(n,\ell)$ is non-negligible in comparison
with~$|\forb_r(n,C_\ell)|$ and were thus able to deduce
Theorem~\ref{thm1}.  The following estimate on $g_r(n, \ell)$ is
proved in~\cite{MuWa}.

\begin{lem}[\cite{MuWa}, Lemma~8]\label{lem08}
  For every $r\ge 4$ and $\ell\ge 3$ there is~$c=c(r,\ell)$ such that
  for all large enough~$n$ we have~$\log g_r(n, \ell)\le cn^{r-1}(\log
  n)^{(r-3)/(r-2)}$. 
\end{lem}

We improve Lemma~\ref{lem08} as follows.

\begin{lem}\label{lem8}
  For every $r\ge 4$ and $\ell\ge 3$ we
  have
  \begin{equation}
    \label{eq:12}
    \log g_r(n, \ell)\le2rn^{r-1}\log\log n
  \end{equation}
  for all large enough~$n$.
\end{lem}

Theorem~\ref{thm2} can be derived from Lemma~\ref{lem8} in the same
way that Theorem~\ref{thm1} is derived from Lemma~\ref{lem08}
in~\cite{MuWa}.  It thus remains to prove Lemma~\ref{lem8}.

\section{Proof of Lemma~\ref{lem8}}

To bound~$g_r(n, \ell)$, we should consider all possible
$(r-1)$-graphs~$G$ and their `valid' edge-colorings.  Let an
$(r-1)$-graph~$G$ be fixed.  The authors of~\cite{MuWa} consider
decompositions of~$G$ into balanced complete $(r-1)$-partite
$(r-1)$-graphs~$G_i$, and obtain good estimates on the number of
edge-colorings of each~$G_i$.  In our proof of Lemma~\ref{lem8}, we
also decompose~$G$ into balanced $(r-1)$-partite $(r-1)$-graphs $G_i$,
but with each~$G_i$ not necessarily complete.  We get our improvement
because certain quantitative aspects of our decomposition are better,
and similar estimates can be shown for the number of edge-colorings of
each~$G_i$.

\begin{defn}[$f_r(n,\ell,G)$]
  \label{def:f_r}
  Let~$r\geq3$ and~$\ell\geq3$ be given and let~$G$ be a balanced
  $(r-1)$-partite $(r-1)$-graph with~$V(G)\subset[n]$.
  Let~$f_r(n,\ell,G)$ be the number of edge-colorings $\chi\:G\to[n]$
  such that~$G^\chi$ is $C_\ell^r$-free.
\end{defn}

\begin{lem}\label{thm5}
  For every $r\ge 4$ and $\ell\ge 3$ there is $c=c(r,\ell)$ such that,
  for any~$G$ as in Definition~\ref{def:f_r}, we have
  \begin{equation}
    \label{eq:4}
    f_r(n,\ell,G)\leq n^{cs^{r-2}}(cs^{r-2})^{e(G)},
  \end{equation}
  where~$s=|V(G)|/(r-1)$.
\end{lem}

We use the following result in the proof of Lemma~\ref{thm5}.

\begin{lem}\cite{MuWa}\label{lem10}
  For every $r\ge 3$ and $\ell \ge 3$ there is~$c=c(r,\ell)$ for which
  the following holds.  Let~$G$ be an $r$-partite $r$-graph with
  vertex classes~$V_1,\dots,V_r$ with $|V_i|=s$ for all $i$.
  If~$e(G) > cs^{r-1}$, then~$G$ contains a loose cycle of
  length~$\ell$.
\end{lem}

\begin{proof}[Proof of Lemma~\ref{thm5}]
  Let $G$ be a balanced $(r-1)$-partite $(r-1)$-graph with each part
  of size~$s$ such that $V(G)\subseteq [n]$.  For any
  edge-coloring~$\chi\:G\to[n]$, let
  $Z=\im\chi=\{z_e\: e\in G\}\subseteq [n]$ be the set of all used
  colors.  We first argue that if $G^*$ is $C_{\ell}^r$-free, then
  $|Z| < (c_{\ref{lem10}}+r)s^{r-2}$, where
  $c_{\ref{lem10}}=c_{\ref{lem10}}(r-1, \ell)$ is the constant from
  Lemma~\ref{lem10}.  Indeed, if $|Z| \ge (c_{\ref{lem10}}+r)s^{r-2}$,
  then
  $|Z\setminus V(G)| \ge (c_{\ref{lem10}}+r)s^{r-2} - s(r-1) >
  c_{\ref{lem10}} s^{r-2}$.
  For each color $v\in Z\setminus V(G)$ pick an edge in~$G$ with
  color~$v$.  We get a subgraph $G'\subseteq G$ that is strongly
  rainbow colored with
  $e(G') = |Z\setminus V(G)| > c_{\ref{lem10}}s^{r-2}$.  By
  Lemma~\ref{lem10}, there is a~$C_{\ell}^{r-1}$ in~$G'$ that is
  strongly rainbow, which contradicts the fact that~$G^*$ is
  $C_{\ell}^r$-free (see Remark~\ref{rem:strongrbw}).

  Let~$c=c(r,\ell)=c_{\ref{lem10}}+r$.  We now estimate the
  number~$f_r(n,\ell,G)$ of valid edge-colorings~$\chi$ the graph~$G$
  may have as follows: choose at most~$cs^{r-2}$ colors and then color
  each edge of~$G$ in all possible ways.  We obtain
  \begin{equation}
    \label{eq:3.2}
    f_r(n,\ell,G)\le n^{cs^{r-2}}(cs^{r-2})^{e(G)},
  \end{equation}
  as required.
\end{proof}

Our next result gives a decomposition of any $r$-graph $G$ into balanced $r$-partite $r$-graphs that are not necessarily complete.

\begin{lem}\label{thm6}
  Suppose $r\ge2$ and
  \begin{equation}
    \label{eq:5}
    1\leq s\leq\(1-{1\over r}\)n.
  \end{equation}
  Then any $n$-vertex $r$-graph $G$ can be decomposed into~$t$
  balanced $r$-partite $r$-graphs $G_i\subseteq K_r(s)$
  $(1\leq i\leq t)$, where $t\le(n/s)^r\lceil c\log n\rceil$ and
  \begin{equation}
    \label{eq:6}
    c=c(r)=\frac{-r}{\log\(1-r!/r^r\)}.    
  \end{equation}
\end{lem}
\begin{proof}
  Generate independently and uniformly $\lceil c\log n\rceil$ random
  $r$-partitions~$\Pi_i=(V_j^i)_{1\leq j\leq r}$ of~$[n]$
  ($1\leq i\leq\lceil c\log n\rceil$).  Thus, $v\in V_j^i$ happens with
  probability~$1/r$ for each~$v$, $i$ and~$j$ independently.  Note
  that each $r$-tuple is `captured' by a given~$\Pi_i$ with
  probability~$r!/r^r$.  Thus the probability that a given $r$-tuple
  is \textit{not} captured by \textit{any} $r$-partition~$\Pi_i$ is at
  most 
  \begin{equation}
    \label{eq:8}
    \(1-\frac{r!}{r^r}\)^{\lceil c\log n\rceil}\leq2^{-r\log n} = n^{-r}.  
  \end{equation}
  The union bound implies that there is a
  collection~$\Pi_i=(V_j^i)_{1\leq j\leq r}$
  ($1\leq i\leq\lceil c\log n\rceil$) of $r$-partitions of~$[n]$
  capturing all $r$-tuples of~$[n]$.

  Let us now fix~$i$ and consider~$\Pi_i$.  We now produce,
  from~$\Pi_i$, at most~$(n/s)^r$ subgraphs~$G'$ of~$G$
  with~$G'\subseteq K_r(s)$, with the collection of such~$G'$ capturing
  every $r$-tuple captured by~$\Pi_i$.  Note that doing this for
  every~$i$ completes the proof of our lemma.  To simplify notation,
  let~$V_j=V_j^i$ ($1\leq j\leq r$).

  Partition each~$V_j$ into blocks~$W_k^j$
  ($1\leq k\leq\lceil|V_j|/s\rceil$) arbitrarily, but
  with~$|W_k^j|\leq s$ for all~$k$.  We now consider all the vectors
  \begin{equation}
    \label{eq:7}
    W(k_1,\dots,k_r)=(W_{k_1}^1,\dots,W_{k_r}^r),
  \end{equation}
  where $1\leq k_j\leq\lceil|V_j|/s\rceil$ for all $1\leq j\leq r$.
  Clearly, each such~$W(k_1,\dots,k_r)$ induces, in a natural way, an
  $r$-partite subgraph~$G(k_1,\dots,k_r)$ of~$G$
  with~$G(k_1,\dots,k_r)\subset K_r(s)$.  Moreover,
  those~$G(k_1,\dots,k_r)$ capture all the $r$-tuples captured
  by~$\Pi_i$.  It now suffices to prove that the number of
  such~$G(k_1,\dots,k_r)$ is at most~$(n/s)^r$.  The number of choices
  we have for~$(k_j)_{1\leq j\leq r}$ is
  \begin{multline}
    \label{eq:9}\quad
    \llceil{|V_1|\over s}\rrceil\dots\llceil{|V_r|\over s}\rrceil
    <{1\over s^r}(|V_1|+s)\dots(|V_r|+s)\\
    \leq{1\over s^r}\Big({1\over r}\sum_{1\leq j\leq
    r}(|V_j|+s)\Big)^r
    \leq\({1\over sr}(n+sr)\)^r
    \leq\(n\over s\)^r,\quad
  \end{multline}
  where we used~\eqref{eq:5} in the last inequality above.
\end{proof}

Now we are ready to prove Lemma~\ref{lem8}.

\begin{proof}[Proof of Lemma~\ref{lem8}]
  Let $G_0$ be an $(r-1)$-graph with $V(G_0)\subseteq [n]$.  In what
  follows, we assume~$n$ is large enough for our inequalities to hold.
  We first count the number of edge-colored $(r-1)$-graphs~$G$ such
  that its underlying $(r-1)$-graph is~$G_0$ and~$G^*$ is
  $C_\ell$-free.  By Lemma~\ref{thm6}, we decompose~$G_0$ into
  balanced $(r-1)$-partite $(r-1)$-graphs $G_1,\dots,G_t$, such that,
  for all $i\in [t]$, each vertex class of~$G_i$ contains
  $s=\lfloor(\log n)^2\rfloor$ vertices and
  $t\le2c_{\ref{thm6}}(n/s)^{r-1}\log n$,
  where~$c_{\ref{thm6}}=c_{\ref{thm6}}(r-1)$ is as given by
  Lemma~\ref{thm6}.  Note that
  $ts^{r-2}\le3c_{\ref{thm6}} n^{r-1}/\log n$ and, since
  $G_1,\dots,G_t$ form a decomposition of~$G_0$, we have
  $\sum_{i=1}^t e(G_i) = e(G_0)$.  Moreover, since~$G^*$ is
  $C_\ell$-free, each $G_i^*$ has to be $C_\ell$-free.  By
  Lemma~\ref{thm5}, the number of valid edge-colorings of~$G_0$ is at
  most
  \begin{multline}
    \label{eq:10}\quad
    \prod_{i=1}^t f_r(n, \ell, G_i)
    \le\prod_{i=1}^tn^{c_{\ref{thm5}}s^{r-2}}(c_{\ref{thm5}}s^{r-2})^{e(G_i)} 
    =n^{c_{\ref{thm5}}ts^{r-2}}(c_{\ref{thm5}}s^{r-2})^{e(G_0)}\\
    \leq n^{3c_{\ref{thm5}}c_{\ref{thm6}}n^{r-1}/\log n}
    (c_{\ref{thm5}}s^{r-2})^{n^{r-1}}.\quad
  \end{multline}
  Since there are at most $2^{n^{r-1}}$ graphs~$G_0$ as above
  and~$s=\lfloor(\log n)^2\rfloor$, summing over all~$G_0$ gives
  \begin{align}
    \log g_r(n,\ell)
    &\leq\log\big(2^{n^{r-1}}
    n^{3c_{\ref{thm5}}c_{\ref{thm6}}n^{r-1}/\log n}
    (c_{\ref{thm5}}s^{r-2})^{n^{r-1}}
    \big)\nonumber\\
    &\le(3c_{\ref{thm5}}c_{\ref{thm6}}+1)n^{r-1}
    +(\log c_{\ref{thm5}}+(r-2)\log s)n^{r-1}\nonumber\\
    &\le (3c_{\ref{thm5}}c_{\ref{thm6}}+1)n^{r-1}
    +(\log c_{\ref{thm5}}+2(r-2)\log\log n)n^{r-1}\nonumber\\
    &\leq2rn^{r-1}\log\log n,    \label{eq:11}
  \end{align}
  where the last inequality follows for all large enough~$n$.
\end{proof}

\section*{Acknowledgement}
We thank Dhruv Mubayi and Lujia Wang for helpful discussions.

\bibliographystyle{amsplain}
\bibliography{Sep2016,alternative}

\begin{bibdiv}
\begin{biblist}

\bib{BDDLS}{article}{
      author={Balogh, J\'ozsef},
      author={Das, Shagnik},
      author={Delcourt, Michelle},
      author={Liu, Hong},
      author={Sharifzadeh, Maryam},
       title={Intersecting families of discrete structures are typically
  trivial},
        date={2015},
        ISSN={0097-3165},
     journal={J. Combin. Theory Ser. A},
      volume={132},
       pages={224\ndash 245},
         url={http://dx.doi.org/10.1016/j.jcta.2015.01.003},
      review={\MR{3311345}},
}

\bib{bollobas98:_hered}{inproceedings}{
      author={Bollob\'as, B\'ela},
       title={Hereditary properties of graphs: asymptotic enumeration, global
  structure, and colouring},
        date={1998},
   booktitle={Proceedings of the {I}nternational {C}ongress of
  {M}athematicians, {V}ol. {III} ({B}erlin, 1998)},
       pages={333\ndash 342},
      review={\MR{1648167}},
}

\bib{EKRo}{article}{
      author={Erd{\H{o}}s, P.},
      author={Kleitman, D.~J.},
      author={Rothschild, B.~L.},
       title={Asymptotic enumeration of {$K_{n}$}-free graphs},
        date={1976},
       pages={19\ndash 27. Atti dei Convegni Lincei, No. 17},
      review={\MR{0463020}},
}

\bib{FrKu}{article}{
      author={{Frankl}, Peter},
      author={{Kupavskii}, Andrey},
       title={{Counting intersecting and pairs of cross-intersecting
  families}},
        date={2017-01},
     journal={ArXiv e-prints},
      eprint={1701.04110},
}

\bib{FurediJiang}{article}{
      author={F{\"u}redi, Z.},
      author={Jiang, T.},
       title={Hypergraph {T}ur\'an numbers of linear cycles},
        date={2014},
        ISSN={0097-3165},
     journal={J. Combin. Theory Ser. A},
      volume={123},
       pages={252\ndash 270},
         url={http://dx.doi.org/10.1016/j.jcta.2013.12.009},
      review={\MR{3157810}},
}

\bib{KMV}{article}{
      author={Kostochka, A.},
      author={Mubayi, D.},
      author={Verstra{\"e}te, J.},
       title={Tur\'an problems and shadows {I}: {P}aths and cycles},
        date={2015},
        ISSN={0097-3165},
     journal={J. Combin. Theory Ser. A},
      volume={129},
       pages={57\ndash 79},
         url={http://dx.doi.org/10.1016/j.jcta.2014.09.005},
      review={\MR{3275114}},
}

\bib{morris16:_c}{article}{
      author={Morris, Robert},
      author={Saxton, David},
       title={The number of {$C_{2\ell}$}-free graphs},
        date={2016},
        ISSN={0001-8708},
     journal={Adv. Math.},
      volume={298},
       pages={534\ndash 580},
         url={http://dx.doi.org/10.1016/j.aim.2016.05.001},
      review={\MR{3505750}},
}

\bib{MuWa}{article}{
      author={{Mubayi}, Dhruv},
      author={{Wang}, Lujia},
       title={{The number of triple systems without even cycles}},
        date={2017-01},
     journal={ArXiv e-prints},
      eprint={1701.00269},
}

\bib{proemel01:_asymp}{article}{
      author={Pr\"omel, Hans~J\"urgen},
      author={Steger, Angelika},
      author={Taraz, Anusch},
       title={Asymptotic enumeration, global structure, and constrained
  evolution},
        date={2001},
        ISSN={0012-365X},
     journal={Discrete Math.},
      volume={229},
      number={1-3},
       pages={213\ndash 233},
         url={http://dx.doi.org/10.1016/S0012-365X(00)00211-9},
        note={Combinatorics, graph theory, algorithms and applications},
      review={\MR{1815608}},
}

\end{biblist}
\end{bibdiv}

\end{document}